\documentclass[11pt,a4paper]{article}

\usepackage{amsmath,amssymb,mathtools}
\usepackage{amsthm}
\usepackage{booktabs}
\usepackage{array}
\usepackage[a4paper,margin=1in]{geometry}
\usepackage[numbers,sort&compress]{natbib}
\usepackage{microtype}
\usepackage[hidelinks]{hyperref}

\newtheorem{theorem}{Theorem}[section]
\newtheorem{proposition}{Proposition}[section]
\newtheorem{lemma}{Lemma}[section]
\newtheorem{corollary}{Corollary}[section]
\theoremstyle{remark}

\newcommand{\tstar}{t_*}
\newcommand{\Afail}{\text{A-fail}}
\newcommand{\Wfail}{\text{W-fail}}
\newcommand{\Ipar}{\mathcal I}

\title{Initialization-dependent BFGS trial rates for tilted absolute values}
\author{Qiuyu Chen\\[0.4em]
\small Department of Computer Science, Shanghai Jiao Tong University\\
\small Shanghai, China}
\date{}

\begin{document}

\maketitle

\begin{abstract}
Lewis and Overton conjectured in 2008 that, for BFGS applied to the one-dimensional tilted absolute value with their Armijo--Wolfe line search, every nonterminating execution converges at a trial-normalized rate determined only by the tilt parameter and independent of the initial data. We prove that this initialization-independent rate assertion is false. In fact, for every tilt parameter in an explicit open interval, we construct two nonterminating executions with different sharp trial-normalized convergence rates. The proof reduces the line search to a scale-free state consisting of the iterate sign and the zero-crossing stepsize, and verifies two periodic state cycles by elementary rational inequalities. This yields an explicit family of counterexamples and shows that the asymptotic trial-normalized behavior of nonsmooth BFGS can depend essentially on the initialization even for this canonical one-dimensional model.
\end{abstract}

\section{Introduction}\label{sec:introduction}

The BFGS method is designed for smooth optimization, yet it is often effective on nonsmooth objectives. Lewis and Overton~\cite{LewisOverton2008,LewisOverton2013} documented this behavior for a range of nonsmooth examples and isolated several low-dimensional models amenable to exact analysis. Subsequent work has clarified different parts of the picture: Lewis and Zhang~\cite{LewisZhang2015} analyzed representative success and failure mechanisms, Guo and Lewis~\cite{GuoLewis2018} established nonsmooth variants of Powell-type BFGS convergence results, Asl and Overton~\cite{AslOverton2021} studied limited-memory BFGS on a nonsmooth convex class, and Gebken~\cite{Gebken2026} recently derived further convergence mechanisms for nonsmooth quasi-Newton methods. These results also show how much of the nonsmooth behavior remains governed by the interaction between the quasi-Newton update and the line search.

A particularly transparent model was introduced by Lewis and Overton in their analysis of BFGS with a doubling--bisection Armijo--Wolfe line search. They considered
\begin{equation}\label{eq:fu}
        f_u(x)=\max\{x,-ux\},\qquad u>0,
\end{equation}
and conjectured that, for fixed $u$, every initialization either reaches the minimizer or converges R-linearly with a rate $r(u)$, measured by the number of line-search trials and depending only on $u$~\cite[Conjecture~4.1]{LewisOverton2008}. The same discussion emphasizes independence of the initial point, while the numerical rate used by Lewis and Overton is normalized by function trials rather than accepted iterations.

This paper gives an explicit parametric family of counterexamples. Let
\begin{equation}\label{eq:I}
        \Ipar=\left(\frac95,\frac{13}{6}\right).
\end{equation}
For every $u\in\Ipar$ and every Wolfe parameter $c_2\in(0,1)$, we construct two different initializations whose line-search dynamics are periodic up to scale. In each case the scale-free state returns exactly after one period, while the full state $(x,H)$ contracts by a fixed factor. The two periodic regimes have different trial-normalized asymptotic rates, so the same tilt parameter $u$ admits distinct sharp convergence rates depending on the initialization. Moreover, the construction is rational in $u$; hence every rational $u\in\Ipar$ yields a pair of finite rational certificates, giving an infinite family of explicit counterexamples to initialization independence.

The remainder of the paper is organized as follows. Section~\ref{sec:dynamics} derives the scale-free line-search dynamics and fixes the trial-normalized rate convention. Section~\ref{sec:cycles} constructs the two parameterized periodic families and computes their exact rates. Section~\ref{sec:consequence} records the resulting failure of initialization independence in the Lewis--Overton conjecture. A Lean formalization of the main results has been carried out and is publicly available at \url{https://github.com/canghaimeng/lewis-overton-bfgs-rate}.

\section{Trial-normalized rate and scale-free dynamics}\label{sec:dynamics}

We first isolate the rate notion used below. Let $(v_k)$ be a positive sequence tending to zero and let $(n_k)$ be positive integers with $n_k\to\infty$. Define
\begin{equation}\label{eq:rate-def}
        \rho(v;n)=\limsup_{k\to\infty}v_k^{1/n_k}.
\end{equation}
When $n_k=k$, this is the usual root rate. When $v_k$ denotes accepted function values, we take $n_k$ to be the cumulative number of line-search trials.

\begin{lemma}[Exact R-linear factor]\label{lem:rate}
For $(v_k)$ and $(n_k)$ as above,
\begin{equation}\label{eq:rate-inf}
\begin{aligned}
 \rho(v;n)=\inf\Bigl\{r>0:\ &\exists C>0\text{ such that}\\
 &v_k\le C r^{n_k}\text{ for all sufficiently large }k\Bigr\}.
\end{aligned}
\end{equation}
\end{lemma}

\begin{proof}
If $v_k\le C r^{n_k}$ eventually, then $v_k^{1/n_k}\le C^{1/n_k}r$, and hence $\rho(v;n)\le r$. Conversely, if $r>\rho(v;n)$, then $v_k^{1/n_k}<r$ for all sufficiently large $k$, so $v_k<r^{n_k}$ eventually. Taking the infimum over admissible factors proves \eqref{eq:rate-inf}.
\end{proof}

This is the trial-clock version of the exact R-linear factor represented by the upper supporting slope on a semilog plot in the Lewis--Overton numerical discussion~\cite{LewisOverton2008}.

Away from the origin, the ordinary gradient of \eqref{eq:fu} is
\begin{equation}\label{eq:gradient}
 g(x)=
 \begin{cases}
 1,&x>0,\\
 -u,&x<0.
 \end{cases}
\end{equation}
Let $H>0$ be the one-dimensional inverse-Hessian approximation and put
\[
        p=-Hg(x).
\]
The search direction points toward the origin. Define the zero-crossing stepsize
\begin{equation}\label{eq:tstar}
        \tstar=-\frac{x}{p}>0,
\end{equation}
so that $x+\tstar p=0$.

For the line-search function
\[
        h(t)=f_u(x+tp)-f_u(x),
\]
write $s=g(x)p=-Hg(x)^2<0$. As in the tilted-absolute-value analysis of Lewis and Overton, we set the Armijo parameter $c_1=0$ and allow any weak Wolfe parameter $c_2\in(0,1)$. Thus a trial $t>0$ satisfies Armijo when $h(t)<0$, and it satisfies weak Wolfe when $h$ is differentiable at $t$ and $h'(t)>c_2s$.

The line search is initialized with $\alpha=0$, $\beta=+\infty$, and $t=1$. An Armijo failure sets $\beta=t$, a Wolfe failure sets $\alpha=t$, and an acceptable trial terminates the search. The next trial is $(\alpha+\beta)/2$ when $\beta<+\infty$ and $2\alpha$ otherwise~\cite[Algorithm~2.6]{LewisOverton2008}. Every tested stepsize, including the accepted one, counts as one trial.

\begin{proposition}[Acceptance window]\label{prop:window}
A trial $t>0$ is accepted if and only if
\begin{equation}\label{eq:accept-window}
        \tstar<t<\tstar\vartheta(x),
        \qquad
        \vartheta(x)=
        \begin{cases}
        1+1/u,&x>0,\\
        1+u,&x<0.
        \end{cases}
\end{equation}
The trial point is zero exactly when $t=\tstar$. In particular, the acceptance interval is independent of $c_2\in(0,1)$.
\end{proposition}

\begin{proof}
For $0<t<\tstar$, the trial point has the same sign as $x$, so $h(t)=ts<0$ and $h'(t)=s$. Since $s<0$ and $c_2<1$, weak Wolfe fails. At $t=\tstar$, Armijo holds but $h$ is not differentiable.

For $t>\tstar$, the trial point lies on the opposite ray. If $x>0$, then $p=-H$, $\tstar=x/H$, and
\[
        h(t)=tuH-(1+u)x,
\]
so Armijo is equivalent to $t<\tstar(1+1/u)$. If $x<0$, then $p=uH$, $\tstar=-x/(uH)$, and
\[
        h(t)=tuH+(1+u)x,
\]
so Armijo is equivalent to $t<\tstar(1+u)$. On the opposite ray, $h'(t)=uH>0>c_2s$, so weak Wolfe holds. This gives \eqref{eq:accept-window}.
\end{proof}

After an accepted nonzero trial, the one-dimensional inverse-BFGS update is the secant update
\begin{equation}\label{eq:secant}
        H_+=\frac{x_+-x}{g(x_+)-g(x)},
        \qquad x_+=x+tp.
\end{equation}

\begin{proposition}[Accepted-step map]\label{prop:map}
Suppose $t$ is accepted and set $\lambda=t/\tstar$. Then $x_+$ has the opposite sign from $x$ and
\begin{align}
 |x_+|&=(\lambda-1)|x|,\label{eq:xplus}\\
 H_+&=
 \begin{cases}
 tH/(1+u),&x>0,\\
 tuH/(1+u),&x<0,
 \end{cases}\label{eq:Hplus}\\
 t_{*,+}&=
 \begin{cases}
 \dfrac{(\lambda-1)(1+u)}{\lambda u},&x>0,\\[5pt]
 \dfrac{(\lambda-1)(1+u)}{\lambda},&x<0,
 \end{cases}\label{eq:tstarplus}\\
 \frac{f_u(x_+)}{f_u(x)}&=
 \begin{cases}
 u(\lambda-1),&x>0,\\
 (\lambda-1)/u,&x<0.
 \end{cases}\label{eq:value-ratio}
\end{align}
Hence $(\operatorname{sign}x,\tstar)$ is a complete scale-free state for the accepted-step dynamics.
\end{proposition}

\begin{proof}
Acceptance places $x_+$ strictly across the origin. If $x>0$, then $p=-H$, $\tstar=x/H$, and $x_+=x(1-\lambda)$. If $x<0$, then $p=uH$, $\tstar=-x/(uH)$, and $x_+=(\lambda-1)|x|$. This proves \eqref{eq:xplus}. Substitution in \eqref{eq:secant} gives \eqref{eq:Hplus}, and inserting \eqref{eq:xplus}--\eqref{eq:Hplus} into the definition of the next zero-crossing stepsize gives \eqref{eq:tstarplus}. Formula \eqref{eq:value-ratio} follows by evaluating $f_u$ on the two rays.

Finally, replacing $(x,H)$ by $(\gamma x,\gamma H)$ with $\gamma>0$ multiplies the search direction, every trial point, and the updated pair by $\gamma$, while leaving $\tstar$ and all line-search comparisons unchanged. Thus the sign and $\tstar$ determine the scale-free evolution.
\end{proof}

\section{Two parameterized periodic families}\label{sec:cycles}

For $u\in\Ipar$, define four zero-crossing stepsizes
\begin{align}
 \tau_A^-(u)&=\frac{(u+1)(2u-1)}{4u^2+5u+4},
 &\tau_A^+(u)&=\frac{3(u+1)(u+2)}{4u^2+5u+4},\label{eq:tauA}\\
 \tau_B^+(u)&=\frac{(u+1)(u+2)}{u^2+u+1},
 &\tau_B^-(u)&=\frac{(u-1)(u+1)}{2(u^2+u+1)}.\label{eq:tauB}
\end{align}
The superscript records the sign of the corresponding iterate. The line-search words below are encoded by four chains of strict inequalities.

\begin{lemma}[Uniform line-search words]\label{lem:words}
For every $u\in\Ipar$,
\begin{align}
 \tau_A^-&<\frac12<(u+1)\tau_A^-<1,\label{eq:wordAminus}\\
 1&<\tau_A^+<\frac32<\frac{u+1}{u}\tau_A^+<2,\label{eq:wordAplus}\\
 1&<\tau_B^+<2<\frac{u+1}{u}\tau_B^+,\label{eq:wordBplus}\\
 \tau_B^-&<\frac12<(u+1)\tau_B^-<1.\label{eq:wordBminus}
\end{align}
\end{lemma}

\begin{proof}
Set
\[
        P(u)=u^3-2u-2,
        \qquad
        Q(u)=-2u^3+u^2+5u+5.
\]
On $\Ipar$, $P$ is increasing and $Q$ is decreasing, with
\[
        P\!\left(\frac95\right)=\frac{29}{125}>0,
        \qquad
        Q\!\left(\frac{13}{6}\right)=\frac5{27}>0.
\]
After clearing the positive denominators, the comparisons that determine the endpoints of the common parameter window reduce to $P(u)>0$ or $Q(u)>0$. The remaining numerators can be written as
\begin{align*}
4u^3+2u^2-5u-6&=4P(u)+2u^2+3u+2,\\
-2u^3+3u^2+6u+4&=Q(u)+2u^2+u-1,\\
5u^3-2u^2-7u-6&=5P(u)-2u^2+3u+4,\\
-u^3+2u^2+3u+2&=\tfrac12 Q(u)+\tfrac12(3u^2+u-1),\\
-u^3+u^2+3u+3&=\tfrac12 Q(u)+\tfrac12(u^2+u+1).
\end{align*}
All added terms are positive on $\Ipar$; for the only non-immediate one,
$-2u^2+3u+4>0$ follows from $u<13/6$, since this quadratic is decreasing on $\Ipar$ and has value $10/9$ at $13/6$. The other comparisons have numerators $3(u+2)$, $-u^2+4u+2$, $3u(2u-1)$, $2u+1$, $u(u-1)$, and $u+2$, which are positive on $\Ipar$. This proves the four chains.
\end{proof}

Define the contraction factors
\begin{equation}\label{eq:q}
        q_A(u)=\frac{3u}{4(u+1)^2},
        \qquad
        q_B(u)=\frac{u}{(u+1)^2}.
\end{equation}
We use the normalized starting states
\begin{equation}\label{eq:starts}
        (x_0^A,H_0^A)=\bigl(-u\tau_A^-(u),1\bigr),
        \qquad
        (x_0^B,H_0^B)=\bigl(\tau_B^+(u),1\bigr).
\end{equation}
The following theorem is the main result. Here $\Afail$ and $\Wfail$ denote an Armijo failure and a weak Wolfe failure, respectively.

\begin{theorem}[Two exact periodic families]\label{thm:main}
Let $u\in\Ipar$ and $c_2\in(0,1)$. Starting from the states in \eqref{eq:starts}, the Lewis--Overton line search generates two nonterminating executions with the following properties.

\begin{enumerate}
\item The $A$-execution has scale-free cycle
\begin{equation}\label{eq:Acycle}
        \bigl(-,\tau_A^-(u)\bigr)
        \longmapsto
        \bigl(+,\tau_A^+(u)\bigr)
        \longmapsto
        \bigl(-,\tau_A^-(u)\bigr).
\end{equation}
One period consists of
\[
\begingroup
\setlength{\arraycolsep}{8pt}
\renewcommand{\arraystretch}{1.35}
\begin{array}{c|ccccc}
\text{trial} &1&\frac12&1&2&\frac32\\ \hline
\text{branch} &\Afail&\text{accept}&\Wfail&\Afail&\text{accept}
\end{array}
\endgroup
\]
and the full state $(x,H)$ is multiplied by $q_A(u)$ after the five trials.

\item The $B$-execution has scale-free cycle
\begin{equation}\label{eq:Bcycle}
        \bigl(+,\tau_B^+(u)\bigr)
        \longmapsto
        \bigl(-,\tau_B^-(u)\bigr)
        \longmapsto
        \bigl(+,\tau_B^+(u)\bigr).
\end{equation}
One period consists of
\[
\begingroup
\setlength{\arraycolsep}{8pt}
\renewcommand{\arraystretch}{1.35}
\begin{array}{c|cccc}
\text{trial} &1&2&1&\frac12\\ \hline
\text{branch} &\Wfail&\text{accept}&\Afail&\text{accept}
\end{array}
\endgroup
\]
and the full state $(x,H)$ is multiplied by $q_B(u)$ after the four trials.
\end{enumerate}

No trial point in either execution equals zero. If $(z_j^A)$ and $(z_j^B)$ denote all trial function values, then
\begin{equation}\label{eq:zrec}
        z_{j+5}^A=q_A(u)z_j^A,
        \qquad
        z_{j+4}^B=q_B(u)z_j^B.
\end{equation}
Consequently their exact trial-normalized rates are
\begin{equation}\label{eq:rates}
        \rho_A(u)=q_A(u)^{1/5},
        \qquad
        \rho_B(u)=q_B(u)^{1/4},
\end{equation}
and
\begin{equation}\label{eq:rate-order}
        \rho_B(u)<\rho_A(u).
\end{equation}
The same two rates are obtained from accepted function values when they are indexed by cumulative trial count.
\end{theorem}

\begin{proof}
At the $A$-starting state, the acceptance interval is
\[
        \bigl(\tau_A^-,(u+1)\tau_A^-\bigr).
\]
By \eqref{eq:wordAminus}, the trial $1$ fails Armijo and $1/2$ is accepted. Proposition~\ref{prop:map} gives the next state
\begin{equation}\label{eq:A1}
        (x_1^A,H_1^A)=
        \left(u\left(\frac12-\tau_A^-\right),\frac{u}{2(u+1)}\right).
\end{equation}
Its zero-crossing stepsize is
\begin{equation}\label{eq:Aidentity1}
        (1-2\tau_A^-)(u+1)=\tau_A^+.
\end{equation}
The next acceptance interval is therefore
\[
        \left(\tau_A^+,\frac{u+1}{u}\tau_A^+\right).
\]
By \eqref{eq:wordAplus}, the trials $1$, $2$, and $3/2$ give a Wolfe failure, an Armijo failure, and acceptance, respectively. A direct use of \eqref{eq:xplus}--\eqref{eq:Hplus}, together with
\begin{equation}\label{eq:Aidentity2}
        \left(1-\frac{2\tau_A^+}{3}\right)\frac{u+1}{u}=\tau_A^-,
\end{equation}
then yields
\begin{equation}\label{eq:Aclosure}
        (x_2^A,H_2^A)=q_A(u)(x_0^A,H_0^A).
\end{equation}
Thus \eqref{eq:Acycle} and the five-trial word repeat indefinitely.

For the $B$-starting state, the acceptance interval is
\[
        \left(\tau_B^+,\frac{u+1}{u}\tau_B^+\right).
\]
By \eqref{eq:wordBplus}, the trial $1$ fails Wolfe and $2$ is accepted, giving
\begin{equation}\label{eq:B1}
        (x_1^B,H_1^B)=
        \left(\tau_B^+-2,\frac{2}{u+1}\right).
\end{equation}
Its zero-crossing stepsize is
\begin{equation}\label{eq:Bidentity1}
        \left(1-\frac{\tau_B^+}{2}\right)\frac{u+1}{u}=\tau_B^-.
\end{equation}
The next acceptance interval is
\[
        \bigl(\tau_B^-,(u+1)\tau_B^-\bigr).
\]
By \eqref{eq:wordBminus}, the trial $1$ fails Armijo and $1/2$ is accepted. Using
\begin{equation}\label{eq:Bidentity2}
        (1-2\tau_B^-)(u+1)=\tau_B^+,
\end{equation}
we obtain
\begin{equation}\label{eq:Bclosure}
        (x_2^B,H_2^B)=q_B(u)(x_0^B,H_0^B).
\end{equation}
This proves \eqref{eq:Bcycle} and the four-trial word.

All comparisons in Lemma~\ref{lem:words} are strict, so none of the tested stepsizes equals the contemporaneous zero-crossing stepsize. Hence all trial function values are positive. Scale equivariance from Proposition~\ref{prop:map} now gives \eqref{eq:zrec}.

If a positive sequence satisfies $z_{j+N}=qz_j$, then writing $j=mN+r$ gives $z_j=q^m z_r$ and therefore
\[
        \lim_{j\to\infty} z_j^{1/j}=q^{1/N}.
\]
This proves \eqref{eq:rates}; the accepted-value statement follows identically after replacing the index by cumulative trial count.

Finally, equality of the two rates would imply, after taking twentieth powers,
\[
        q_A(u)^4=q_B(u)^5.
\]
Using \eqref{eq:q}, this is equivalent to
\[
        81(u+1)^2=256u,
\]
or
\[
        81u^2-94u+81=0.
\]
Its discriminant is $94^2-4\cdot81^2=-17408<0$, so the quadratic is positive for all real $u$. Hence \eqref{eq:rate-order} holds.
\end{proof}

\begin{corollary}[Infinite rational certificates]\label{cor:rational}
For every rational $u\in\Ipar$, both periodic executions in Theorem~\ref{thm:main} have rational initial states, rational trial points, and rational contraction factors. In particular, the theorem provides infinitely many explicit rational pairs of executions with different trial-normalized rates at the same value of $u$.
\end{corollary}

\begin{proof}
All quantities in \eqref{eq:tauA}, \eqref{eq:tauB}, \eqref{eq:q}, and \eqref{eq:starts} are rational functions of $u$ with nonzero denominators on $\Ipar$. The interval $\Ipar$ contains infinitely many rational numbers.
\end{proof}

\section{Consequence for the Lewis--Overton conjecture}\label{sec:consequence}

\begin{corollary}\label{cor:conjecture}
For every $u\in\Ipar$, there are two nonterminating initializations whose exact R-linear factors, measured by line-search trials, are the unequal values in \eqref{eq:rates}.
\end{corollary}

\begin{proof}
Fix $u\in\Ipar$. Theorem~\ref{thm:main} gives two nonterminating executions for the same objective $f_u$ and the same line-search parameters. By Lemma~\ref{lem:rate}, their rates in \eqref{eq:rates} are the infimal R-linear factors under the trial clock. Since \eqref{eq:rate-order} shows that these factors are different, no rate depending only on $u$ can equal the exact trial-normalized rate of every nonterminating initialization. The initialization-independent rate assertion in Lewis and Overton's Conjecture~4.1~\cite{LewisOverton2008} is false.
\end{proof}

The construction identifies the mechanism behind this dependence: the doubling--bisection line search admits distinct periodic words in the scale-free state $(\operatorname{sign}x,\tstar)$, and each word carries its own contraction per trial. This reduction turns the rate question into a finite calculation on explicit rational maps and suggests that other periodic words can be studied by the same approach.

\bibliographystyle{plainnat}
\bibliography{references}

\end{document}